\documentclass{amsart}

\usepackage{amssymb}
\usepackage{graphicx}
\usepackage{color}
\newtheorem{theorem}{Theorem}[section]
\newtheorem{lemma}[theorem]{Lemma}

\newtheorem{corollary}[theorem]{Corollary}
\theoremstyle{definition}

\newtheorem{example}[theorem]{Example}
\newtheorem{proposition}[theorem]{Proposition}
\theoremstyle{remark}
\newtheorem{remark}[theorem]{Remark}

\numberwithin{equation}{section}

\begin{document}
\title[Perturbations of vector fields on the torus]{Global hypoellipticity for perturbations of complex vector fields on the torus}


\author{M. V. Bartmeyer}
\address{\small Departamento de Matem\'atica, Universidade Estadual de Maring\'a, Av. Colombo, 5790, Maring\'a, 87020-900, PR, Brazil.
}
\email{pg55473@uem.br}
\thanks{The first author was supported by the Coordenação de Aperfeiçoamento de Pessoal de Nível Superior – Brazil (CAPES) – Finance Code 001}

\author{P. L. Dattori da Silva}
\address{\small Departamento de Matem\'atica, Instituto de Ci\^encias Matem\'aticas e de\break Computa\c c\~ao,
Universidade de S\~ao Paulo, Caixa Postal 668, S\~ao Carlos, SP 13560-970, Brazil}
\email{dattori@icmc.usp.br}
\thanks{The second author was supported in part by São Paulo Research Foundation (FAPESP), Brazil (Process Number $\#$2024/08416-6 and $\#$2024/12753-8) and National Council for Scientific and Technological Development – CNPq, Brazil (grant $\#$313581/2021-5)}

\author{R. B. Gonzalez}
\address{\small Departamento de Matem\'atica, Universidade Estadual de Maring\'a, Av. Colombo, 5790, Maring\'a, 87020-900, PR, Brazil.
}
\email{rbgonzalez@uem.br}



\begin{abstract}We apply Kr\"{o}necker's approximation theorem to measure (in a topological sense) a set of constants which turn a vector field into a non-globally hypoelliptic operator. We present situations in which this set is a discrete enumerable (hence, meager) subset of the real line, and we also show that this set may be a dense $\mathcal{G}_\delta$ subset of the complex numbers (hence, nonmeager), which produces a contrast to a known result stating that this set has null Lebesgue measure. 
\end{abstract}

\maketitle

\section{Introduction}

One of the main questions in partial differential equations is the regularity of the solutions of a given equation. Considering periodic solutions on a Euclidean space, Fourier series leads us to the study of algebraic equations; consequently, certain approximation theorems of number theoretic nature appear in the study.

Given $n\geq1,$ we consider $\mathbb{T}^n\simeq\mathbb{R}^n/2\pi\mathbb{Z}^n$ the $n-$dimensional torus and we denote by $\mathcal{C}^\infty(\mathbb{T}^n)$ the space of the smooth and complex-valued functions on $\mathbb{T}^n.$ Endowing $\mathcal{C}^\infty(\mathbb{T}^n)$ with the usual metrizable topology, we have a Fr\'{e}chet space whose dual is denoted by $\mathcal{D}'(\mathbb{T}^n),$ which is the usual space of periodic distributions. 

A partial differential operator $P:\mathcal{C}^\infty(\mathbb{T}^n)\rightarrow\mathcal{C}^\infty(\mathbb{T}^n)$ is globally hypoelliptic (GH for short) if the conditions $u\in \mathcal{D}'(\mathbb{T}^n)$ and $Pu\in \mathcal{C}^\infty(\mathbb{T}^n)$ imply that $u\in\mathcal{C}^\infty(\mathbb{T}^n).$ 

As we know from \cite{GW}, the real vector field $D_t+\alpha D_x$ ($\alpha\in\mathbb{R}$) is globally hypoelliptic if and only if $\alpha$ is a non-Liouville irrational number. Since the set of Liouville numbers is a dense $\mathcal{G}_\delta$ (a comeager subset of $\mathbb{R},$ which is nonmeager by Baire's Theorem), it follows that the set of the real numbers $\alpha$ for which $D_t+\alpha D_x$ is globally hypoelliptic is considered a negligible subset.   

Once we know the behavior of a differential operator, a natural question is to analyse which happens with its perturbations. With this in mind, given $\alpha\in\mathbb{C},$ we consider \[\mathcal{N}_1=\{\lambda\in\mathbb{R}; \ \ D_t+\alpha D_x-\lambda \ \ \mbox{is not GH}\}.\]

In \cite{B} the author studied the size of $\mathcal{N}_1$ in a topological sense. Some of the results in \cite{B} are described below.
 
Assume that $\alpha$ is a rational number and write $\alpha=p/q,$ in which $p\in\mathbb{Z},$ $q\in\mathbb{N}$ and $p/q$ is an irreducible fraction (when $\alpha=0$ we set $p=0$ and $q=1$). Recall that the vector field $D_t+(p/q) D_x$ is not globally hypoelliptic (see \cite{GW}). On the other hand, by \cite{B}, it follows that the operator $D_t+(p/q) D_x-\lambda$ is not globally hypoelliptic if and only if $\lambda\in q^{-1}\mathbb{Z}.$ In this situation, $\mathcal{N}_1=\lambda\in q^{-1}\mathbb{Z}$ is an enumerable discrete subset of $\mathbb{R};$ hence, a negligible subset.  Assuming now that $\alpha$ is an irrational number, it follows from \cite{B} (see \cite{B} - Proposition 3.5) that the set $\mathcal{N}_1$ becomes a dense $\mathcal{G}_\delta.$    

A next attempt could be to consider complex perturbations. Given $\alpha\in\mathbb{C},$ we set \[\mathcal{M}_1=\{\lambda\in\mathbb{C}; \ \ D_t+\alpha D_x-\lambda \ \ \mbox{is not GH}\}.\]

Notice that, if $\alpha\in\mathbb{R},$ then $\mathcal{M}_1=\mathcal{N}_1.$ If $\alpha\in\mathbb{C}\setminus\mathbb{R},$ then the vector field $D_t+\alpha D_x$ is GH (indeed, it is elliptic) and $\mathcal{N}_1=\emptyset=\mathcal{M}_1.$ Hence, in dimension two there is no gain when we allow $\lambda$ being a complex number, since $\mathcal{M}_1=\mathcal{N}_1$ in all the situations.    

We summarize the above comments into the following:

\begin{theorem}[see \cite{B}]\label{thm1} Given $\alpha\in\mathbb{C},$ we have $\mathcal{M}_1=\mathcal{N}_1$ and
\begin{itemize}
\item[i)] if $\alpha\in\mathbb{C}\setminus\mathbb{R},$ then $\mathcal{M}_1=\emptyset;$
\item[ii)] if $\alpha\in\mathbb{R}\setminus\mathbb{Q},$ then $\mathcal{M}_1$ is a dense $\mathcal{G}_\delta$ (in particular a comeager set)   
\item[iii)] if $\alpha=p/q$ (either an irreducible fraction or $p=0$ and $q=1$) then $\mathcal{M}_1=q^{-1}\mathbb{Z}$ which is a meager set. 
\end{itemize}
\end{theorem}

\begin{remark}\label{rk1}As we may see from $i)-$Theorem \ref{thm1}, the global hypoellipticity of $D_t+\alpha D_x,$ $\alpha\in\mathbb{C}\setminus\mathbb{R},$ is strong enough so that it cannot be affected by a constant perturbation. Indeed, it cannot be affected by a perturbation of order zero (see Proposition 3.1 in \cite{B}). We will show that this phenomenon is a particularity of the dimension two.
\end{remark}

The purpose of this article is to extend the previous result by considering functions which depend on more variables.

We work with complex vector fields of the type 
\[L_N=D_t+\sum_{j=1}^{N}\alpha_jD_{x_j}, \ \ (t,x_1,\ldots,x_N)\in\mathbb{T}^{N+1},\] in which $\alpha_j\in\mathbb{C}.$ For $\lambda\in\mathbb{C}$ we set
\[P_{N,\lambda}=L_N-\lambda.\]

Since we are interested in dimensions higher than two, we assume $N\geq2.$

The sets under study are 
\[\mathcal{N}_N=\{\lambda\in\mathbb{R}; P_{N,\lambda} \ \mbox{is not GH}\}\] and
\[\mathcal{M}_N=\{\lambda\in\mathbb{C}; P_{N,\lambda} \ \mbox{is not GH}\}.\]

Notice that $\mathcal{N}_N\subset\mathcal{M}_N.$ 

By following the approach in \cite{B} we will prove that $\mathcal{M}_N$ is a $\mathcal{G}_\delta$ subset of $\mathbb{C}$ (see Proposition \ref{propp1}). In addition, in Theorem \ref{thm6} we use a version of an approximation result known as Kronecker's approximation theorem (which is a generalization of Dirichlet's approximation theorem) in order to present a characterization to the density of $\mathcal{M}_N$.

The techniques in \cite{GW} allow us to state that $P_{N,\lambda}$ is globally hypoelliptic if and only if the following algebraic condition holds:
\begin{itemize}
\item[] there exist positive constants $C,M,R$ such that \begin{equation}\label{equ1}|\tau+\alpha_1\xi_1+\cdots+\alpha_N\xi_N-\lambda|\geq C|(\tau,\xi_1,\ldots,x_N)|^{-M},\end{equation} for all $(\tau,\xi)\in\mathbb{Z}\times\mathbb{Z}^N$ such that $|(\tau,\xi)|\geq R,$ where $|(\tau,\xi_1,\ldots,x_N)|=|\tau|+|\xi_1|+\cdots+|\xi_N|$.
\end{itemize}

A consequence of the above characterization is that $\mathcal{N}_N=\mathcal{M}_N$ whenever all the coefficients $\alpha_j$ are real numbers. In this case, we will see in Proposition \ref{propp2} that $\mathcal{M}_N$ may be either a discrete enumerable (meager) subset of $\mathbb{R}$ or a dense $\mathcal{G}_\delta$ subset of $\mathbb{R}.$  

In the presence of a coefficient which is not a real number, we obtain a contrast with the two dimensional case. Indeed, we will see that we may lose the global hypoellipticity by adding a perturbation (see Proposition \ref{prop2}). Moreover, in the presence of at least two coefficients which are not real numbers, we stress that we may obtain a dense $\mathcal{G}_\delta$ (nonmeager in the complex plane) set of non-globally hypoelliptic perturbations (see Theorem \ref{thm6} and Example \ref{ex1}; compare to Remark \ref{rk1}). 

It is also worth pointing out that in higher dimensions, the presence of a non-real coefficient implies that $\mathcal{M}_N$ is different from $\mathcal{N}_N$ (see Proposition \ref{prop2} and Theorem \ref{thm6}).

In Section \ref{sec3} we apply these results (concerning constant coefficients vector fields) to measure the size of a set of non-globally hypoelliptic perturbations of a globally hypoelliptic tube-type vector field. We present a condition which yields that this set is a dense $\mathcal{G}_\delta$ subset of the space of smooth periodic functions.

It is an interesting fact that the addition of a term of order zero may destroy the property of global hypoellipticity of operators of principal type, contrary to that happens with the usual (local) hypoellipticity (see \cite{Treves}). It is also interesting the existence of a large set (in a topological sense) of terms of order zero which destroy the global hypoellipticity. This fact is in contrast to Theorem 3 in \cite{BZ} which yields that the Lebesgue measure of $\mathcal{M}_1$ is zero (small in measure sense). Indeed, the Lebesgue measure of $\mathcal{M}_N$ is zero (see Theorem 5.1.9 in \cite{Ng}).

\section{Main results}
We are interested in the global hypoellipticity of \[P_{N,\lambda}=L_N-\lambda,\] in which \[L_N=D_t+\sum_{j=1}^{N}\alpha_jD_{x_j}, \ \ (t,x_1,\ldots,x_N)\in\mathbb{T}^{N+1},\] with $\alpha_j\in\mathbb{C}.$ 

We denote by \[\rho_\lambda(\tau,\xi)=\tau+\alpha_1\xi_1+\cdots+\alpha_N\xi_N-\lambda, \ \ (\tau,\xi)\in\mathbb{Z}\times\mathbb{Z}^N,\] the symbol of the differential operator $P_{N,\lambda}.$ 

Motivated by \cite{B} we define, for $j\in\mathbb{N}$ and $(\tau,\xi)\in\mathbb{Z}\times\mathbb{Z}^N\setminus\{0\},$ 
\[\mathcal{M}(j,\tau,\xi)=\left\{\lambda\in\mathbb{C}; |\rho_\lambda(\tau,\xi)|<(|\xi|+|\tau|)^{-j}\right\}.\]

We also set \[\mathcal{M}_N(j)=\bigcup_{\substack{(\tau,\xi)\in\mathbb{Z}\times\mathbb{Z}^N,\\ |(\tau,\xi)|>1}}\mathcal{M}(j,\tau,\xi)\] and \[\mathcal{M}_{N,\infty}=\bigcap_{j\in\mathbb{N}}\mathcal{M}_N(j).\]

Notice that each $\mathcal{M}(j,\tau,\xi)$ is an open ball, which implies that $\mathcal{M}_{N,\infty}$ is a $\mathcal{G}_\delta$ subset of $\mathbb{C}.$ In addition, the characterization \eqref{equ1} implies that $\mathcal{M}_N\subset \mathcal{M}_{N,\infty}.$ 

To give a first description to $\mathcal{M}_N$ we will use the following technical result, whose proof is a straightforward manipulation of condition \eqref{equ1}.

\begin{lemma}\label{lem1}If $\lambda\in\mathbb{Z}+\alpha_1\mathbb{Z}+\cdots+\alpha_N\mathbb{Z},$ then $L_N$ is globally hypoelliptic if and only if $P_{N,\lambda}$ is globally hypoelliptic.    
\end{lemma}

The first description to $\mathcal{M}_N$ is given by the following result:

\begin{proposition}\label{propp1}If $L_N$ is globally hypoelliptic, then $\mathcal{M}_{N}=\mathcal{M}_{N,\infty}\cap(\mathbb{Z}+\alpha_1\mathbb{Z}+\ldots+\alpha_N\mathbb{Z})^{c}.$ On the other hand, if $L_N$ is not globally hypoelliptic, then $\mathcal{M}_{N}=\mathcal{M}_{N,\infty}.$ In any case, $\mathcal{M}_{N}$ is a $\mathcal{G}_\delta$ subset of $\mathbb{C}.$
\end{proposition}
\begin{proof}We know that $\mathcal{M}_N\subset \mathcal{M}_{N,\infty}$ and we start this proof by showing that  $\mathcal{M}_{N,\infty}\cap(\mathbb{Z}+\alpha_1\mathbb{Z}+\cdots+\alpha_N\mathbb{Z})^{c}\subset\mathcal{M}_N.$ 

Pick $\lambda\in\mathcal{M}_{N,\infty}\cap(\mathbb{Z}+\alpha_1\mathbb{Z}+\cdots+\alpha_N\mathbb{Z})^{c}.$ For each $j\in\mathbb{N},$ there exist $(\tau_j,\xi(j))\in\mathbb{Z}\times\mathbb{Z}^{N}$ with $|\tau_j|+|\xi(j)|> 1$ and \begin{equation}\label{eq4}0<|\rho_\lambda(\tau_j,\xi(j))|<|(\tau_j,\xi(j)|^{-j}.\end{equation}  

Suppose, by way of contradiction, that $\lambda\not\in\mathcal{M}_{N}.$ Then there exist positive constants $C,M,R$ such that \begin{equation}\label{eq3}|(\tau,\xi)|^{-M}C\leq|\rho_\lambda(\tau,\xi)|,\end{equation} whenever $|(\tau,\xi)|\geq R.$ By adjusting the constant $C,$ we may assume that \eqref{eq3} is satisfied for all $(\tau,\xi)\in\mathbb{Z}\times\mathbb{Z}^N\setminus\{0\}$ such that $\rho_\lambda(\tau,\xi)\neq0.$ 

From \eqref{eq4} and \eqref{eq3} we obtain 
\[0<C\leq|(\tau_j,\xi(j))|^{M-j}\leq 2^{M-j}, \ \ \mbox{for} \ \ j\in\mathbb{N} \ \ \mbox{large enough}.\] 

The above estimate is clearly a contradiction. Therefore, we have $\mathcal{M}_{N,\infty}\cap(\mathbb{Z}+\alpha_1\mathbb{Z}+\cdots+\alpha_N\mathbb{Z})^{c}\subset\mathcal{M}_N.$

Supposing that $L_N$ is globally hypoelliptic, then Lemma \ref{lem1} implies that $\mathcal{M}_{N}\subset\mathcal{M}_{N,\infty}\cap(\mathbb{Z}+\alpha_1\mathbb{Z}+\cdots+\alpha_N\mathbb{Z})^{c}.$ Therefore, $\mathcal{M}_{N}=\mathcal{M}_{N,\infty}\cap(\mathbb{Z}+\alpha_1\mathbb{Z}+\cdots+\alpha_N\mathbb{Z})^{c}$ provided that $L$ is globally hypoelliptic.

On the other hand, supposing that $L_N$ is not globally hypoelliptic, by Lemma \ref{lem1} we have $\mathbb{Z}+\alpha_1\mathbb{Z}+\cdots+\alpha_N\mathbb{Z}\subset\mathcal{M}_{N}$ and, consequently, $\mathcal{M}_{N,\infty}\subset\mathcal{M}_N\cup[\mathcal{M}_{N,\infty}\cap(\mathbb{Z}+\alpha_1\mathbb{Z}+\cdots+\alpha_N\mathbb{Z})^c]\subset\mathcal{M}_{N}.$ Therefore, $\mathcal{M}_{N}=\mathcal{M}_{N,\infty}$ whenever $L_N$ is not globally hypoelliptic.

We have previously noticed that $\mathcal{M}_{N,\infty}$ is a $\mathcal{G}_\delta$ subset of $\mathbb{C}.$ Since the same holds true for $(\mathbb{Z}+\alpha_1\mathbb{Z}+\cdots+\alpha_N\mathbb{Z})^c,$ it follows that $\mathcal{M}_N$ is a $\mathcal{G}_\delta$ subset of $\mathbb{C}.$
\end{proof}

In certain cases we may explicitly determine who the set $\mathcal{M}_N$ is.

\begin{proposition}\label{propp2}Suppose that $\alpha_j\in\mathbb{R}$ for each $j=1,\ldots,N.$ In addition,  if there exists $k$ such that $\alpha_k\in\mathbb{R}\setminus\mathbb{Q},$ then $\mathcal{M}_N$ is a comeager dense $\mathcal{G}_\delta$ subset of $\mathbb{R}.$ On the other hand, writing $\alpha_j=p_j/q_j$ (irreducible fraction or $p_j=0$ and $q_j=1$) and $Q_j\doteq q_1\ldots q_{j-1}q_{j+1}\ldots q_N,$ we obtain 
\[\mathcal{M}_N=\frac{\gcd(q_1\ldots q_N,p_1Q_1,\ldots,p_NQ_N)}{q_1\ldots q_N}\mathbb{Z},\] which is a discrete enumerable (meager) $\mathcal{G}_\delta$ subset of $\mathbb{R}.$ 
\end{proposition}
\begin{proof}Notice that $\mathcal{M}_N=\mathcal{N}_N\subset\mathbb{R}$ whenever $\alpha_j\in\mathbb{R}$ for each $j.$ If $\alpha_k\in\mathbb{R}\setminus\mathbb{Q}$ then $\mathbb{Z}+\alpha_k\mathbb{Z}$ is dense in $\mathbb{R}$ (a well know consequence of Dirichlet's approximation theorem) and $(\mathbb{Z}+\alpha_k\mathbb{Z})\setminus\{0\}\subset \mathcal{M}_N.$ It follows from Proposition \ref{propp1} that $\mathcal{M}_N$ is a dense $\mathcal{G}_\delta$ subset of $\mathbb{R}.$

We now proceed to the situation in which each $\alpha_j$ is a rational number. In this case, defining $Q=q_1\ldots q_N$ we have
\[\tau+\sum_{j=1}^{N}\alpha_j\xi_j-\lambda=\frac{1}{Q}\left[(\tau-\lambda)Q+\sum_{j=1}^{N}p_j\xi_jQ_j\right],\] for all $(\tau,\xi)\in\mathbb{Z}\times\mathbb{Z}^{N}.$ Since \[\tau Q+\sum_{j=1}^{N}p_j\xi_jQ_j\in\mathbb{Z},\] condition \eqref{equ1} will fail to hold if and only if there exist infinitely many indices $(\tau,\xi)$ such that \[(\tau-\lambda)Q+\sum_{j=1}^{N}p_j\xi_jQ_j=0;\] this last condition holds if and only if $Q\lambda\in\gcd(Q,p_1Q_1,\ldots,p_Nq_N)\mathbb{Z}.$


\end{proof}

\begin{remark}If $\alpha_k\in\mathbb{Q}$ for at least one $k,$ then $L_N$ is not globally hypoelliptic. In particular, $L_N$ is not globally hypoelliptic when $\alpha_j\in\mathbb{Q}$ for all $j.$ However, Proposition \ref{propp2} implies that there exists a large (its complementary is negligible) set of complex numbers $\lambda$ for which the perturbations $P_{N,\lambda}$ are globally hypoelliptic.   

\end{remark}

As we see in the previous result, $\mathcal{M}_N$ may be a discrete enumerable subset of the real line. In the presence of a non-real coefficient, we will present a result which gives a condition so that the set $\mathcal{M}_N$ reaches the opposite extreme, that is, $\mathcal{M}_N$ will be a dense $\mathcal{G}_\delta$ subset of $\mathbb{C}.$ 

By using Baire's theorem and Proposition \ref{propp1} we see that the density of $\mathcal{M}_N$ follows from the density of $\mathcal{M}_{N,\infty}.$ 

\begin{lemma}\label{lem2} The set $\mathcal{M}_{N,\infty}$ is dense in $\mathbb{C}$ if and only if $\mathbb{Z}+\alpha_1\mathbb{Z}+\cdots+\alpha_N\mathbb{Z}$ is dense in $\mathbb{C}.$
\end{lemma}

\begin{proof}Suppose that  $\mathcal{M}_{N,\infty}$ is dense. Given $z\in\mathbb{C}$ there exists a sequence $\lambda_n$ in $\mathcal{M}_{N,\infty}$ which converges to $z.$ Since $\lambda_n\in\mathcal{M}_N(n),$ there exists $(\tau_n,\xi(n))\in\mathbb{Z}\times\mathbb{Z}^N$ such that $|(\tau_n,\xi(n))|>1$ and \[|\tau_n+\alpha_1\xi(n)_1+\cdots+\alpha_N\xi(n)_N-\lambda_n|<|(\tau_n,\xi(n))|^{-n}\leq2^{-n}.\] By triangular inequality, the sequence $\tau_n+\alpha_1\xi(n)_1+\cdots+\alpha_N\xi(n)_N$ converges to $z.$ 

The converse is trivial since $\mathbb{Z}+\alpha_1\mathbb{Z}+\cdots+\alpha_N\mathbb{Z}\subset\mathcal{M}_{N,\infty}.$

\end{proof}

We are now in position to state and prove the main result of this section.

With a fixed index $k\in\{1,\ldots,N\}$ (recall that $N\geq2$) we will use the following matrix $A=(a_{ij})_{2\times N-1},$ in which
\[a_{1j}=\Re\alpha_j\Im\alpha_k-\Re\alpha_k\Im\alpha_j \ \ \mbox{and} \ \ a_{2j}=\Im\alpha_j, \ \ j=1,\ldots,k-1,k+1,\ldots,N.\]

\begin{theorem}\label{thm6}Suppose that there exists $k\in\{1,\ldots,N\}$ such that $\Im\alpha_k\neq0.$ The set $\mathcal{M}_N$ is a dense $\mathcal{G}_\delta$ subset of $\mathbb{C}$ if and only if \[\left\{r\in\mathbb{Q}^2; \Im\alpha_k^{-1}A^{T}r\in\mathbb{Q}^{N-1} \right\}=\{(0,0)\}.\] 
\end{theorem}
\begin{proof}Under the assumption that $\Im\alpha_k\neq0,$ the set $B=\{(1,0), (\Re\alpha_k,\Im\alpha_k)\}$ is a basis for $\mathbb{R}^2.$ The coordinates of an element $z\in\mathbb{Z}+\alpha_1\mathbb{Z}+\ldots+\alpha_N\mathbb{Z}$ are given by \[z=\left[\ell+\sum_{j=1,\ldots,k-1,k+1,\ldots,N}\left(\Re\alpha_j-\frac{\Re\alpha_k\Im\alpha_j}{\Im\alpha_k}\right)m_j, n+\sum_{j=1,\ldots,k-1,k+1,\ldots,N}\frac{\Im\alpha_j}{\Im\alpha_k}m_j\right]_B,\] in which $\ell,m_j,n\in\mathbb{Z}.$ 

We denote the first linear map by \[T_1(m)=\sum_{j=1,\ldots,k-1,k+1,\ldots,N}\left(\Re\alpha_j-\frac{\Re\alpha_k\Im\alpha_j}{\Im\alpha_k}\right)m_j\] and the second we denote by \[T_2(m)=\sum_{j=1,\ldots,k-1,k+1,\ldots,N}\frac{\Im\alpha_j}{\Im\alpha_k}m_j\]

It follows that $\mathbb{Z}+\alpha_1\mathbb{Z}+\ldots+\alpha_N\mathbb{Z}$ is dense in $\mathbb{C}$ if and only if for each $\zeta=[x,y]_B$ and $\epsilon>0$ there exist $\ell,n\in\mathbb{Z}$ and $m\in\mathbb{Z}^{N-1}$ such that \[|\ell+T_1(m)-x|<\epsilon \ \ \mbox{and} \ \ |n+T_2(m)-y|<\epsilon.\]  

Applying the Kronecker's approximation theorem (see \cite{Vorselen}; also, \cite{Kronecker}) we see that $\mathbb{Z}+\alpha_1\mathbb{Z}+\ldots+\alpha_N\mathbb{Z}$ is dense in $\mathbb{C}$ if and only if \[\left\{r\in\mathbb{Q}^2; \Im\alpha_k^{-1}A^{T}r\in\mathbb{Q}^{N-1} \right\}=\{(0,0)\}.\]

We finally apply Proposition \ref{propp1} and Lemma \ref{lem2} to complete the proof.
\end{proof}

Notice that the initial assumption in Theorem \ref{thm6} is that one coefficient is not real. However, if exactly one coefficient is not real, then the second line of the matrix $A$ vanishes identically and the condition presented in Theorem \ref{thm6} is not satisfied, that is, $\mathcal{M}_N$ is not dense in $\mathbb{C}.$ 

When $N=2$ and $\Im\alpha_1\neq0,$ the condition given in Theorem \ref{thm6} for the density of $\mathcal{M}_{2}$ reduces to the following one: for each $(r_1,r_2)\in\mathbb{Q}^2$ such that
\[\left(\Re\alpha_2-\Re\alpha_1\frac{\Im\alpha_2}{\Im\alpha_1}\right)r_1+\frac{\Im\alpha_2}{\Im\alpha_1}r_2\in\mathbb{Q}\]
we must have $r_1=r_2=0.$ This condition holds true if and only if \begin{equation}\label{equ4}1, \ \  \Re\alpha_2-\Re\alpha_1\frac{\Im\alpha_2}{\Im\alpha_1} \ \ \mbox{and} \ \ \frac{\Im\alpha_2}{\Im\alpha_1}\end{equation} are linearly independent over the integers. If $\Im\alpha_2=0,$ then this condition is not satisfied and $\mathcal{M}_{2}$ is not dense in $\mathbb{C}.$ On the other hand, if also $\Im\alpha_2\neq0,$ then we could use Theorem \ref{thm6} to obtain another condition for the density of $\mathcal{M}_2,$ which would be the linear independence (over the integers) for the three numbers \[1, \ \  \Re\alpha_1-\Re\alpha_2\frac{\Im\alpha_1}{\Im\alpha_2} \ \ \mbox{and} \ \ \frac{\Im\alpha_1}{\Im\alpha_2}.\] This last condition is equivalent to the previous one in \eqref{equ4}    

\begin{example}\label{ex1}
If $\Im\alpha_2\neq0$ and $\Im\alpha_1/\Im\alpha_2$ is an irrational non-Liouville number, then by \cite{GW} we know that $\Im\alpha_1D_x+\Im\alpha_2D_y$ is globally hypoelliptic on $\mathbb{T}^2_{(x,y)}.$ This implies that $D_t+\alpha_1D_x+\alpha_2D_y$ is globally hypoelliptic on $\mathbb{T}^3.$ In addition, if $\Re\alpha_2\in\mathbb{Q}$ and $1,\Re\alpha_1, \Im\alpha_1/\Im\alpha_2$ are linearly independent over $\mathbb{Z},$ then by Theorem \ref{thm6} the set $\mathcal{M}_2$ is a dense $\mathcal{G}_\delta$ subset of $\mathbb{C}$ and, consequently, there exists a nonmeager set of perturbations which are not globally hypoelliptic.  In particular, the vector field $D_t+(\sqrt{2}+i\sqrt{3})D_x+(1+i)D_y$ is globally hypoelliptic and lots of its perturbations by constant are not.
\end{example}

Next, Propositions \ref{prop3} and \ref{prop4} will give more details about $\mathcal{M}_N$ (and, consequently $\mathcal{N}_N$) in the presence of non-real coefficients.

\begin{proposition}\label{prop3}
Suppose that $\Im\alpha_k\neq0$ and $\Im\alpha_j=0$ for $j\in\{1,\ldots,N\}\setminus\{k\}.$ Then $\mathcal{M}_N$ is contained in a collection of horizontal lines in $\mathbb{C},$ 
\[\mathcal{M}_N\subset\bigcup_{m\in\mathbb{Z}}t_m,\] in which \[t_m=\{z\in\mathbb{C}; \Im z=m \Im\alpha_k\}.\]
Moreover, $\mathcal{M}_N\cap t_m$ is a discrete enumerable subset of $t_m$ whenever $\Re\alpha_j\in\mathbb{Q}$ for each $j\in\{1,\ldots,N\}\setminus\{k\}.$ On the other hand, if there exists $j\in\{1,\ldots,N\}\setminus\{k\}$ such that $\Re\alpha_j$ is an irrational number, then $\mathcal{M}_N\cap t_m$ is dense in $t_m.$ 
\end{proposition}
\begin{proof}We only sketch an idea to the proof. 

Notice that the imaginary part of the symbol is $\Im\alpha_k\xi_k-\Im\lambda.$ This is far from zero if $\lambda$ is not in the lines $t_m.$ If $\lambda\in t_m,$ then the imaginary part could be zero and the real part becomes 
\[\tau+\left(\sum_{j\in\{1,\ldots,N\}\setminus\{k\}}\alpha_j\xi_{j}\right)-\left(\Re\lambda-m\frac{\Re\alpha_k}{\Im\alpha_k}\right).\]

It follows that $M_{N}\cap t_m$ is the preimage of the set \[\mathcal{N}_{N}=\left\{\gamma\in\mathbb{R}; D_t+\left(\sum_{j\in\{1,\ldots,N\}\setminus\{k\}}\alpha_jD_{x_j}\right)-\gamma \ \ \mbox{is not GH on} \ \ \mathbb{T}\times\mathbb{T}^{N-1}\right\}\] by the map $\Theta_m(x+iy)=m(\Re\alpha_k/\Im\alpha_k) - x.$

We now apply Proposition \ref{propp2} to complete the proof.
\end{proof}

Next two results, in dimension $N=2,$ give more details about the sets of perturbations which are non-globally hypoelliptic.

\begin{corollary}\label{cor3}Suppose that $\alpha_1\in\mathbb{R}$ and $\alpha_2\in\mathbb{C}\setminus\mathbb{R}.$ Then, the set
\[
\mathcal{N}_2=\left\{\lambda\in\mathbb{R}; \ D_t+\alpha_1D_x+\alpha_2D_y-\lambda \ \mbox{is not GH on} \ \mathbb{T}^3_{(t,x,y)} \right\}
\]
reduces to 
\[\left\{\lambda\in\mathbb{R}; \ D_t+\alpha_1D_x-\lambda \ \mbox{is not GH on} \ \mathbb{T}^2_{(t,x)} \right\}\] and \[\displaystyle\mathcal{M}_2=\bigcup_{m\in\mathbb{Z}}\theta_m^{-1}(\mathcal{N}_2),\] in which $\theta_m:\mathbb{C}\rightarrow\mathbb{R}$ is given by $\theta_m(x+iy)=(\Re\alpha_2/\Im\alpha_2)m-x.$ In addition, if $\alpha_1\in\mathbb{Q},$ then both $\mathcal{N}_2$ and $\mathcal{M}_2$ are discrete enumerable sets. On the other hand, if $\alpha_1\in\mathbb{R}\setminus\mathbb{Q},$ then $\mathcal{N}_2$ is a dense $\mathcal{G}_\delta$ subset of $\mathbb{R}$ and $\mathcal{M}_2$ is a $\mathcal{G}_\delta$ subset of $\mathbb{C},$ which is dense on each horizontal line belonging to $\{z\in\mathbb{C}; \Im z\in\Im\alpha_2\mathbb{Z}\}.$
\end{corollary}

\begin{proposition}\label{prop2} 
Suppose that $\alpha_1,\alpha_2\in\mathbb{C}\setminus\mathbb{R}$ and that $\Im\alpha_1/\Im\alpha_2$ belongs to $\mathbb{Q}.$ Write $\Im\alpha_1/\Im\alpha_2=p/q$ (irreducible fraction). We have:
\begin{itemize}
\item[(i)] If $\Re\alpha_1q-\Re\alpha_2p=\tilde{p}/\tilde{q}$ (irreducible or $\tilde{p}=0$ and $\tilde{q}=1$), then $\mathcal{N}_2=\tilde{q}^{-1}\mathbb{Z}$ and $\mathcal{M}_2$ is a subset of the union of the lines \begin{equation}\label{equ5}\ell_m=\{\lambda\in\mathbb{C}; \ \Im\lambda=m(\Im\alpha_2/q)\} \ \ m\in\mathbb{Z};\end{equation} in addition, for each $m\in\mathbb{Z},$ the set $\ell_m\cap\mathcal{M}_2$ is a discrete enumerable subset of $\ell_m.$

\item[(ii)] If $\Re\alpha_1q-\Re\alpha_2p\in\mathbb{R}\setminus\mathbb{Q},$ then $\mathcal{N}_2$ is a dense $\mathcal{G}_\delta$ subset of $\mathbb{R},$ while $\mathcal{M}_2$ is a subset of the lines $\ell_m,$ $m\in\mathbb{Z},$ such that $\ell_m\cap\mathcal{M}_2$ is a dense $\mathcal{G}_\delta$ subset of $\ell_m.$  
\end{itemize}
\end{proposition}
\begin{proof}Given $\lambda\in\mathbb{R},$ we have \[|\tau+\alpha_1\xi+\alpha_2\eta-\lambda|\geq|\Im\alpha_2|\left|\frac{\Im\alpha_1}{\Im\alpha_2}\xi+\eta\right|=|q^{-1}\Im\alpha_2||p\xi+q\eta|.\]

Since $p\xi+q\eta$ is an integer number, we obtain \[|\tau+\alpha_1\xi+\alpha_2\eta-\lambda|\geq|q^{-1}\Im\alpha_2|>0,\] for all $(\xi,\eta)\not\in (q,-p)\mathbb{Z}.$    

It follows that $\lambda\in\mathcal{N}_2$ if and only if there exists a sequence $(\tau_j,\xi_j,\eta_j)=(\tau_j,m_jq,-m_jp)\in\mathbb{Z}^3,$ with $m_j\in\mathbb{Z},$ such that $|\tau_j|+|m_j|(|p|+|q|)\geq j$ and \[|\tau_j+m_j(q\Re\alpha_1-p\Re\alpha_2)-\lambda|<\frac{1}{(|\tau_j|+|m_j|[|p|+|q|])^{j}}.\] 

The above estimate is equivalent to the non-global hypoellipticity of the operator $D_t+(q\Re\alpha_1-p\Re\alpha_2)D_x-\lambda$ on $\mathbb{T}^2_{(t,x)}.$ Hence, we may write 
\begin{equation}\label{eq2}\mathcal{N}_2=\left\{\lambda\in\mathbb{R}; \ D_t+(q\Re\alpha_1-p\Re\alpha_2)D_x-\lambda \ \mbox{is not GH on} \ \mathbb{T}^2_{(t,x)} \right\}.\end{equation}

For a complex number $\lambda,$ the imaginary part of the symbol $\rho_\lambda(\tau,\xi,\eta)=\tau+\alpha_1\xi+\alpha_2\eta-\lambda$ may be written by
\[q^{-1}\Im\alpha_2\left(p\xi+q\eta-\frac{q\Im\lambda}{\Im\alpha_2}\right).\]

Hence, if we pick $\lambda$ on a line $\ell_m$ (see \eqref{equ5}), then the symbol does not approximate to zero if $(\xi,\eta)$ is not a solution to the Diophantine equation $p\xi+q\eta=m.$ Picking $(\xi_m,\eta_m)$ a solution, the other solutions are given by $(\xi_m,\eta_m)+k(q,-p),$ $k\in\mathbb{Z}.$  On this set of solutions, the symbol $\rho_\lambda$ becomes \[|\tau+k(q\Re\alpha_1-p\Re\alpha_2)+\Re\alpha_1\xi_m+\Re\alpha_2\eta_m-\Re\lambda|.\] It follows that $P_\lambda$ is not globally hypoelliptic if and only if $D_t+(q\Re\alpha_1-p\Re\alpha_2)D_x-(\Re\lambda-\Re\alpha_1\xi_m-\Re\alpha_2\eta_m)$ is not globally hypoelliptic on $\mathbb{T}^2_{(t,x)}.$
 
If $\psi_m:\ell_m\rightarrow\mathbb{R}$ is given by \[\psi_m\left(x+i\frac{m\Im\alpha_2}{q}\right)=x-\Re\alpha_1\xi_m-\Re\alpha_2\eta_m,\] then $\psi_m$ is an isometry and  
\begin{equation}\label{eq1}\ell_m\cap\mathcal{M}_2=\psi_m^{-1}(\mathcal{N}_{2});\end{equation} consequently, \[\mathcal{M}_2=\bigcup_{m\in\mathbb{Z}}\psi^{-1}_m(\mathcal{N}_{2}).\] 

If $\Re\alpha_1q-\Re\alpha_2p\in\mathbb{Q},$ we write $\Re\alpha_1q-\Re\alpha_2p=\tilde{p}/\tilde{q}$ (irreducible or $\tilde{p}=0$ and $\tilde{q}=1$).  By \eqref{eq2} and Theorem \ref{thm1} it follows that $\mathcal{N}_2=\tilde{q}^{-1}\mathbb{Z}$ and by \eqref{eq1} we have $\ell_m\cap\mathcal{M}_2$ is a discrete enumerable subset of $\ell_m.$   

On the other hand, if $\Re\alpha_1q-\Re\alpha_2p\in\mathbb{R}\setminus\mathbb{Q},$ then Theorem \ref{thm1} and \eqref{eq2} imply that $\mathcal{N}_2$ is a dense $\mathcal{G}_\delta$ subset of $\mathbb{R}$ and by \eqref{eq1} it follows that $\ell_m\cap\mathcal{M}_2$ is a dense $\mathcal{G}_\delta$ subset of $\ell_m.$

The proof is then completed.
\end{proof}

\begin{example}\label{ex2}If $\alpha_1,\alpha_2\in\mathbb{C}\setminus\mathbb{R}$ and $\lambda\in\mathcal{N}_2,$ then we may use \eqref{equ1} to obtain a sequence $(\xi_j,\eta_j)\in\mathbb{Z}^2$ such that $|\xi_j|+|\eta_j|\rightarrow\infty$ and \[|\Im\alpha_2|^{-1}\left|\frac{\Im\alpha_1}{\Im\alpha_2}\xi_j+\eta_j\right|\leq(|\xi_j|+|\eta_j|)^{-j}.\]

The above estimate implies that $\Im\alpha_1/\Im\alpha_2$ is either a rational number or it is an irrational Liouville number. Hence, $\mathcal{N}_2=\emptyset$ if $\Im\alpha_1/\Im\alpha_2$ is an irrational non-Liouville number.

On the other hand, if $\Re\alpha_2=0$ and if we pick $\Im\alpha_1/\Im\alpha_2$ a Liouville number such that the pair $(\Re\alpha_1,\Im\alpha_1/\Im\alpha_2)$ is a Liouville vector, then $0\in\mathcal{N}_2.$ 

When $\Re\alpha_1$ and $\Re\alpha_2$ are integer numbers and $\Im\alpha_1/\Im\alpha_2$ is an irrational number, we may use Theorem \ref{thm1} to show that $\mathcal{M}_2$ is a non-enumerable subset of $\mathbb{C}.$ Indeed, for \[\tilde{\mathcal{N}}=\left\{\gamma\in\mathbb{R}; D_y+(\Im\alpha_1/\Im\alpha_2)D_x-\gamma \ \ \mbox{is not GH on} \ \ \mathbb{T}^{2}_{(y,x)}\right\}\] we have $\mathbb{Z}+i\Im\alpha_2\tilde{\mathcal{N}}\subset\mathcal{M}_2.$ By Theorem \ref{thm1}, the set $\tilde{\mathcal{N}}$ is a dense $\mathcal{G}_\delta$ subset of $\mathbb{R}$ (hence, non-enumerable).
\end{example}

\begin{example}
When $N=2,$ the operators may be written by \[L_2=D_t+\alpha_1D_{x}+\alpha_2D_y,\] \[P_{2,\lambda}=L_2-\lambda,\] and they act on $\mathcal{C}^\infty(\mathbb{T}^3).$ In this case, we may summarize some informations about $\mathcal{M}_2$ and $\mathcal{N}_2$ given in the previous results. For instance, $\mathcal{M}_2$ (resp. $\mathcal{N}_2$) is a discrete enumerable subset of $\mathbb{C}$ (resp. $\mathbb{R}$) in the following cases:
\begin{itemize}  
\item the two coefficients are rational numbers;
\item one coefficient is not a real number and the other is a rational number;
\item the two coefficients are not real numbers, $\Im\alpha_1/\Im\alpha_2=p/q$ (irreducible) and $\Re\alpha_1q-\Re\alpha_2p\in\mathbb{Q}.$
\end{itemize}

If one coefficient is not a real number and the other is an irrational number, then $\mathcal{N}_2$ is a dense $\mathcal{G}_\delta$ (in particular non-enumerable) subset of $\mathbb{R}$ and 

The set $\mathcal{N}_2$ is a dense $\mathcal{G}_\delta$ (in particular non-enumerable) subset of $\mathbb{R}$ and $\mathcal{M}_2$ is a non-enumerable $\mathcal{G}_\delta$ subset of $\mathbb{C}$ densely contained in a set of horizontal lines (which do not accumulate) in the following cases:
\begin{itemize}  
\item at least one coefficient is an irrational number;
\item the two coefficients are not a real number, $\Im\alpha_1/\Im\alpha_2=p/q$ (irreducible) and $\Re\alpha_1q-\Re\alpha_2p\in\mathbb{R}\setminus\mathbb{Q}.$
\end{itemize}

Example \ref{ex2} illustrates certain situations in which the two coefficients are not real numbers and $\Im\alpha_1/\Im\alpha_2$ is an irrational number. In this case, we recall that the set $\mathcal{M}_2$ is a dense $\mathcal{G}_\delta$ (in particular non-enumerable) subset of $\mathbb{C}$ if and only if $1,$ $\Re\alpha_1-\Re\alpha_2(\Im\alpha_1/\Im\alpha_2),$ and $\Im\alpha_1/\Im\alpha_2$ are linearly independent over $\mathbb{Z}.$ We now proceed to show that $\mathcal{M}_2$ is contained in a set of parallel lines (which accumulate) when the condition about linearity fails to hold. Indeed, if the condition fails to hold we may pick integers $k,m,n$ such that $m\neq0$ and \[\Re\alpha_1=\Re\alpha_2\frac{\Im\alpha_1}{\Im\alpha_2}-\frac{n}{m}-\frac{k}{m}\frac{\Im\alpha_1}{\Im\alpha_2}.\]

If $\lambda\in\mathcal{M}_2,$ then there exists a sequence $(\tau_j,\xi_j,\eta_j)\in\mathbb{Z}^3$ such that $|(\tau_j,\xi_j,\eta_j)|\rightarrow\infty$ and \[|\tau_j+\alpha_1\xi_j+\alpha_2\eta_j-\lambda|<|(\tau_j,\xi_j,\eta_j)|^{-j}.\]

Splitting the real and imaginary parts of the symbol we write
\begin{align*}\left|\tau_j+\Re\alpha_2\frac{\Im\lambda}{\Im\alpha_2}-\xi_j\left(\frac{n}{m}+\frac{k}{m}\frac{\Im\alpha_1}{\Im\alpha_2}\right)-\Re\lambda\right|\leq \\
\left|\frac{\Re\alpha_2}{\Im\alpha_2}\right|\left|\eta_j\Im\alpha_2+\xi_j\Im\alpha_1-\Im\lambda\right|+|(\tau_j,\xi_j,\eta_j)|^{-j}\leq\\
\left(\left|\frac{\Re\alpha_2}{\Im\alpha_2}\right|+1\right)|(\tau_j,\xi_j,\eta_j)|^{-j}.\end{align*}

The above estimate implies that $\Re\lambda-\Re\alpha_2\frac{\Im\lambda}{\Im\alpha_2}$ belongs to \[\left\{\gamma\in\mathbb{R}; D_t-\left(\frac{\Im\alpha_1}{\Im\alpha_2}\frac{k}{m}+\frac{n}{m}\right)D_x-\gamma \ \ \mbox{is not GH on} \ \ \mathbb{T}^{2}_{(t,x)}\right\}.\] 

Since $\Im\alpha_1/\Im\alpha_2$ is an irrational number, the above set is dense in $\mathbb{R}$ and then the parallel lines passing through its points accumulate. 

\end{example}

We complete this section presenting a generalization to Proposition \ref{prop2}.

\begin{proposition}\label{prop4}
Suppose that there exist two different indices $k,n\in\{1,\ldots,N\}$ such that $\Im\alpha_k\neq0,$ $\Im\alpha_n\neq0,$ $\Im\alpha_n/\Im\alpha_k=p/q$ (irreducible fraction) and $\Im\alpha_j=0$ for $j\in J\doteq\{1,\ldots,N\}\setminus\{k,n\}.$ 

If  $\{q\Re\alpha_k-p\Re\alpha_n,\alpha_j; \ j\in J\}\cap(\mathbb{R}\setminus\mathbb{Q})\neq\emptyset,$ then $\mathcal{N}_N$ is a dense $\mathcal{G}_\delta$ subset of $\mathbb{R}$ and $\mathcal{M}_N$ is a subset of the union of the lines $\ell_m$ (see \eqref{equ5}) and $\ell_m\cap\mathcal{M}_N$ is dense in $\ell_m.$ In the other case, if $\{q\Re\alpha_k-p\Re\alpha_n,\alpha_j; \ j\in J\}\subset\mathbb{Q},$ then $\mathcal{N}_N$ is an enumerable discrete subset of $\mathbb{R},$ while $\mathcal{M}_N$ is a discrete enumerable subset of $\cup_m\ell_m\subset\mathbb{C}.$

\end{proposition}

\section{Applications: GH for operators with variable coefficients}\label{sec3}

In this section we consider the so-called tube-type operators, which are of the form \[L=D_t+\sum_{j=1}^Nc_j(t)D_{x_j},\ \ (t,x_1,\ldots,x_N)\in\mathbb{T}^{N+1},\] in which $c_j$ is a smooth periodic and complex-valued function, $c_j\in\mathcal{C}^{\infty}(\mathbb{T}^1).$  

The global hypoellipticity of operators of tube-type has been largely studied (see \cite{AGKM,BDG1,BDG3,G,Fe}). Once we know the behavior of an operator, it is natural to ask which happens with its perturbations. 

Given $\lambda\in\mathcal{C}^\infty(\mathbb{T}^{N+1})$ we set \[P_\lambda=L-\lambda\] and \[\mathcal{O}=\{\lambda\in\mathcal{C}^\infty(\mathbb{T}^{N+1}); \ \ P_\lambda \ \ \mbox{is not GH}\}.\]

We also set \[\lambda_0=\frac{1}{(2\pi)^{N+1}}\int_{\mathbb{T}^{N+1}}\lambda, \ \ c_{j0}=\frac{1}{2\pi}\int_{0}^{2\pi}c_j, \ \ j=1,\ldots,N.\] 

We will give conditions so that $\mathcal{O}$ contains a dense $\mathcal{G}_\delta$ subset in $\mathcal{C}^\infty(\mathbb{T}^{N+1});$ consequently, for a GH operator $L,$ the set of zero-order perturbations $P_\lambda$ which remain GH may be meager. 

Define \[L_{00}=D_t+\sum_{j=1}^Nc_{j0}D_{x_j},\]  
\[\mathcal{M}_{0N}=\{z\in\mathbb{C}; \ \ L_{00}-z \ \ \mbox{is not GH}\},\]
and \[G=\{\lambda\in\mathcal{C}^\infty(\mathbb{T}^{N+1}); \ \lambda_0\in\mathcal{M}_{0N}\}.\]

\begin{theorem}We have $G\subset\mathcal{O}$ and the following conditions hold:
\begin{itemize}
\item[i)] $G$ is a dense $\mathcal{G}_\delta$ subset of the Fréchet space $F=\{\lambda\in\mathcal{C}^\infty(\mathbb{T}^{N+1}); \ \Im\lambda_0=0\}$ whenever $\Im c_{j0}=0$ (for each $j=1,\ldots,N$) and there exists $k$ such that $c_{k0}\in\mathbb{R}\setminus\mathbb{Q}.$
\item[ii)] $G$ is a dense $\mathcal{G}_\delta$ subset of $\mathcal{C}^\infty(\mathbb{T}^{N+1})$ provided that there exists $k$ such that $\Im c_{k0}\neq0$ and \[\left\{r\in\mathbb{Q}^2; \Im\alpha_k^{-1}C^{T}r\in\mathbb{Q}^{N-1} \right\}=0,\] in which $C=(C_{ij})_{2\times N-1}$ is given by
\[C_{1j}=\Re c_{j0}\Im c_{k0}-\Re c_{k0}\Im c_{j0} \ \ \mbox{and} \ \ C_{2j}=\Im c_{j0}, \ \ j=1,\ldots,k-1,k+1,\ldots,N.\]
 \end{itemize}

\end{theorem}

\begin{proof}Notice that the map $I:\mathcal{C}^{\infty}(\mathbb{T}^{N+1})\rightarrow \mathbb{C}$ given by \[I(\theta)=\theta_0=\frac{1}{(2\pi)^{N+1}}\int_{\mathbb{T}^{N+1}}\theta\] is a continuous non-constant linear functional (hence, an open mapping). In the situation (ii), it follows that $G=I^{-1}(\mathcal{M}_{0N})$ is a dense $\mathcal{G}_\delta$ subset in $\mathcal{C}^{\infty}(\mathbb{T}^{N+1}),$ since by Theorem \ref{thm6} the set $\mathcal{M}_{0N}$ is a dense $\mathcal{G}_\delta$ subset in $\mathbb{C}.$ In the situation i), we know from Proposition \ref{propp2} that $\mathcal{M}_{N0}$ is a dense $\mathcal{G}_\delta$ subset of $\mathbb{R}.$ Picking the restriction $I\vert_F:F\rightarrow\mathbb{R}$ we still have a continuous linear and open mapping such that $G=(I\vert_F)^{-1}(\mathcal{M}_{N0}).$  

We now proceed to prove that $G\subset\mathcal{O}$ whenever $L$ is GH.

%
%

In order to show that $G\subset\mathcal{O},$ it is enough to show that the non-global hypoellipticity of $L_{00}-\lambda_0$ implies that $P_\lambda=L-\lambda$ is not GH.

The non-global hypoellipticity of $L_{00}-\lambda_0$ implies the existence of a sequence $(\tau_,\xi(n))\in\mathbb{Z}\times\mathbb{Z}^N$ such that $|\tau_n|+|\xi(n)|\geq n$ and \[|\tau_n+c_{10}\xi_1(n)+\cdots+c_{N0}\xi_N(n)-\lambda_0|\leq (|\tau_n|+|\xi(n)|)^{-n}.\]

In order to make the notation shorter we set \[C_n(t)=\sum_{j=1}^Nc_{j}(t)\xi_j(n), \ \ \mbox{and} \ \ C_{n0}=\sum_{j=1}^Nc_{j0}\xi_j(n)dt.\]  

Suppose first that there exists a subsequence such that \[|\tau_n+C_{n0}-\lambda_0|=0, \ \ \mbox{for all } n.\] In this case, pick $t_n\in[0,2\pi]$ satisfying \[\Im\left(\int_0^{t_n}C_n(t)-\lambda(t)dt\right)=\max_{t\in[0,2\pi]}\Im\left(\int_0^{t_n}C_n(t)-\lambda(t)dt\right)=M_n\] and define \[\hat{\mu}(t,\xi(n))=e^{-M_n}\exp\left\{-i\int_0^{t}C_n(\tau)-\lambda(\tau)d\tau\right\}\in\mathcal{C}^\infty(\mathbb{T}^1).\]

The distribution \[\mu=\sum_{n=1}^{\infty}\hat{\mu}(t,\xi(n))e^{i\langle x,\xi(n)\rangle}\] is not a smooth function and satisfies $P_\lambda\mu=0;$ hence, $P_\lambda$ is not GH.

Suppose now that there exists a subsequence such that \[|\tau_n+c_{10}\xi_1(n)+\cdots+c_{N0}\xi_N(n)-\lambda_0|>0, \ \ \mbox{for all } n.\] In this case, similar computations to the ones in Lemma 3.1 in \cite{BDG1} implies that there exists a sequence $\eta(n)\in\mathbb{Z}^N$ such that $|\eta(n)|$ is strictly increasing, $|\eta(n)|>n$ and \[|1-e^{-2\pi i[c_{10}\eta_1(n)+\cdots c_{N0}\eta_N(n)-\lambda_0]}|<|\eta(n)|^{-n}, \ \ \mbox{for all} \ \ n.\] 

As before, we set \[F_n(t)=\sum_{j=1}^Nc_{j}(t)\eta_j(n), \ \ F_{n0}=\sum_{j=1}^Nc_{j0}\eta_j(n).\]

Pick $t_n\in[0,2\pi]$ satisfying \[\Im\left(\int_{t_n}^{t}F_n(t)-\lambda(t) dt\right)\leq 0, \ \ \mbox{for all} \ \ t\in[0,2\pi].\]

Passing to a subsequence if necessary, we may assume that $t_n\rightarrow t_0.$ Pick a closed interval $J\subset (0,2\pi)$ such that $t_0\not\in J.$ 

There exists $\phi\in\mathcal{C}_{c}^{\infty}(J,\mathbb{R})$ such that $0\leq\phi\leq1$ and $\int_{0}^{2\pi}\phi>0.$

Set \[E_n=1-e^{-2\pi i[F_{n0}-\lambda_0]}.\]

Notice that the sequence \[\hat{f}(t,\eta(n))=E_n\phi(t)e^{-i\int_{t_n}^{t}\left(F_n(\tau)-\lambda(\tau)\right) d\tau}, \ \ t\in[0,2\pi],\] decays rapid. Hence
\[f(t,x)=\sum_{n=1}^{\infty}\hat{f}(t,\eta(n))e^{i\langle x,\eta(n)\rangle}\in\mathcal{C}^\infty(\mathbb{T}^{N+1}).\]

On the other hand, the sequence \[\hat{u}(t,\eta(n))=E_n^{-1}\int_{0}^{2\pi}\hat{f}(t-s,\eta(n))e^{-i\int_{t-s}^{t}\left(F_n(\tau)-\lambda(\tau)\right) d\tau}ds, \ \ t\in[0,2\pi],\]
increases slowly. Hence, \[u(t,x)=\sum_{n=1}^{\infty}\hat{u}(t,\eta(n))e^{i\langle x,\eta(n)\rangle}\in\mathcal{D}'(\mathbb{T}^{N+1}).\]

Straightforward computations show that $P_\lambda u=f$ and \[|\hat{u}(t_n,\eta(n))|\geq \frac{1}{2}\int_{0}^{2\pi}\phi(t)dt>0.\] 

It follows that $P_\lambda$ is not GH, since $P_\lambda u\in\mathcal{C}^\infty(\mathbb{T}^{N+1})$ but $u\not\in\mathcal{C}^\infty(\mathbb{T}^{N+1}).$

\end{proof}

\end{document}